\documentclass[11pt]{amsart}

\usepackage[margin=1.2in]{geometry}
\usepackage{mathtools}
\usepackage{dsfont}
\usepackage{setspace}

\newtheorem{theorem}{Theorem}[section]
\newtheorem{lemma}[theorem]{Lemma}

\newtheorem{proposition}[theorem]{Proposition}

\theoremstyle{definition}
\newtheorem{definition}[theorem]{Definition}
\newtheorem{example}[theorem]{Example}

\theoremstyle{remark}
\newtheorem{remark}[theorem]{Remark}

\numberwithin{equation}{section}

\begin{document}

\title[A  derivative concept with respect to an arbitrary kernel]{A  derivative concept with respect to an arbitrary kernel  and applications to fractional calculus}

\author[M. Jleli]{Mohamed Jleli}

\address[M. Jleli]{Department of Mathematics, College of Science, King Saud University, P.O. Box 2455, Riyadh, 11451, Saudi Arabia}
\email{jleli@ksu.edu.sa}

\author[M. Kirane]{Mokhtar Kirane}
\address[M. Kirane]{LaSIE, Facult\'e des Sciences et Technologies, Universit\'e de La Rochelle, La Rochelle Cedex France\newline
NAAM Research Group, Department of Mathematics, Faculty of Science, King Abdulaziz University, Jeddah, Saudi Arabia}
\email{mkirane@univ-lr.fr}

\author[B. Samet]{Bessem Samet}
\address[B. Samet]{Department of Mathematics, College of Science, King Saud University, P.O. Box 2455, Riyadh, 11451, Saudi Arabia}
\email{bsamet@ksu.edu.sa}

\subjclass[2010]{47B34; 26A24; 26A33}



\keywords{Conjugate kernels; $k$-derivative; fractional calculus; boundary value problem}

\begin{abstract}
In this paper,  we propose a new concept of  derivative  with respect to an arbitrary  kernel-function.  Several properties related to this new operator, like inversion rules, integration by parts, etc. are studied. In particular, we introduce the notion of conjugate kernels, which will be useful to guaranty that the proposed derivative operator admits a right inverse. The proposed concept includes as special cases Riemann-Liouville fractional derivatives, Hadamard fractional derivatives, and many other fractional operators. Moreover, using our concept, new fractional operators involving certain special functions  are introduced, and some of their properties are studied. Finally, an existence result for a boundary value problem involving the introduced derivative operator is proved.  
\end{abstract}

\maketitle

\section{Introduction}

In ordinary calculus, differentiation and integration are regarded as discrete operations, in the sense that we differentiate or integrate a function once, twice, or any whole number of times. The natural extension is to try to present new definitions for derivatives and integrals with arbitrary real order $\alpha>0$,  in which the standard definitions are just particular cases. This subject is known in the literature as fractional calculus, and  it goes back to Leibniz \cite{LE} in 1695 when, in a note sent to L'Hospital, he discussed the meaning of the derivative of order $\frac{1}{2}$. That note represented in fact the birth of  the theory of integrals and derivatives of an arbitrary order. For more than two centuries,  this theory has been treated as a purely theoretical mathematical field, and many mathematicians, such  as Liouville  \cite{LI}, Gr\"{u}nwald \cite{GR}, Letnikov \cite{LET}, Marchaud \cite{MA} and Riemann \cite{RI}, have developed this field of research by introducing new definitions and studying their most important properties. 
However, in the past decades, this subject has proven to be useful in many   areas of physics and engineering, such as image processing \cite{GH}, fluid mechanics \cite{KU}, viscoelasticity \cite{BA,FA,MAI}, stochastic processes \cite{CO,TI},   pollution  phenomena \cite{JKS}, geology \cite{AB}, thermal conductivity \cite{KJB}, turbulent flows \cite{GNN}, etc.

The classical fractional calculus is based  on the well-known Riemann-Liouville fractional integrals 
$$
(I_a^\alpha f)(x)=\frac{1}{\Gamma(\alpha)}\int_0^x(x-y)^{\alpha-1}f(y)\,dy,\quad \alpha>0
$$ 
and derivatives of order $\alpha$
$$
(D_0^\alpha f)(x)=\left(\frac{d}{dx}\right)^n (I_0^{n-\alpha} f)(x),
$$
where $\Gamma$ is the Gamma function and $n=[\alpha]+1$, or on the  Erd\'elyi-Kober operators as their immediate generalizations. For more details on these fractional operators and their properties, we refer the reader to Kilbas et al. \cite{KST},  Samko et al. \cite{SKM} and Sneddon \cite{SN}. Due to the importance of fractional calculus in applications, different definitions of fractional operators were introduced by many authors, see for example \cite{AG,AL,ABA,CA,GA,HA,HI,KI0}.  On the other hand, one of the important criteria of a fractional derivative is its dependence on a certain parameter. Moreover, one should get standard derivative of integer order for  integer values of the parameter, see for example  \cite{OM}. In \cite{TA}, Tarasov demonstrated that a violation of the Leibniz rule is a characteristic property of derivatives of fractional orders.

Motivated by the above cited works, a new approach on integrals and derivatives with respect to arbitrary kernels is proposed in this paper. We, first, introduce the notion of conjugate kernels. Next, we define left-sided and right-sided $k$-integral operators and study their properties.  Further, we introduce  left-sided and right-sided $k'$-derivative operators, where $\mbox{conj}(k')$, the set of kernel- functions $k$ such that $k$ and $k'$ are conjugate, is supposed to be nonempty. The relation between $k'$-derivatives and $k$-integrals, where $k\in \mbox{conj}(k')$,  is discussed, and many other properties,  like inversion rules, integration by parts, etc. are established. Moreover, using the introduced concepts, new fractional operators are defined and some of their properties are derived.  We give also an existence result for a boundary value problem involving $k'$-derivative.

The rest of the paper is organized as follows. In Section \ref{sec2}, we first define the set of kernels $\mathcal{K}_\omega$ that depends on an arbitrary non-negative weight function $\omega$ satisfying $\omega, \omega^{-1}\in L^\infty([a,b];\mathbb{R})$, $(a,b)\in \mathbb{R}^2$, $a<b$.
Next, we introduce the concept of conjugate kernels. As we shall see later, this notion  is very important in the proof of the inversion relations between integrals and derivatives with respect to arbitrary kernels. In Section \ref{sec3}, we introduce the class of $k$-integral operators $I_a^k$ and $I_b^k$, where $k$ is an arbitrary kernel that belongs to the set $\mathcal{K}_\omega$, and establish some properties related to such operators. In Section \ref{sec4}, we introduce $k'$-derivative operators $D_a^{k'}$ and $D_b^{k'}$, where $\mbox{conj}(k')\neq\emptyset$. Several properties related to such operators are established. Moreover, we show that for particular choices of $k'$, we obtain Riemann-Liouville and Hadamard fractional derivatives. In Section \ref{sec5}, using the introduced concepts, we define new fractional operators involving exponential integral functions and regularized lower Gamma functions. Several properties for such operators are deduced.
Finally, in Section \ref{sec6},  an existence result for a boundary value problem involving $k'$-derivative is established.

\section{Preliminaries}\label{sec2}

Let $\mathbb{R}_+=]0,+\infty[$ and  $(a,b)\in \mathbb{R}^2$ be such that  $a<b$. We denote by $\Omega$ the subset of $[a,b]\times [a,b]$ defined by
$$
\Omega=\{(x,y)\in [a,b]\times [a,b]\mbox{ such that } x> y\}.
$$
We introduce the set of weight functions
$$
\mathcal{W}=\left\{\omega: [a,b]\to \mathbb{R}_+ \mbox{ such that } \omega, \omega^{-1}\in L^\infty([a,b];\mathbb{R})\right\},
$$
where the space $L^\infty([a,b];\mathbb{R})$ is constructed from the space of measurable functions, bounded almost everywhere.

Let $\omega\in \mathcal{W}$ be a fixed weight function. For a given kernel-function $k: \Omega \to \mathbb{R}_+$, let
$$
F_k: y\mapsto \int_y^b k(x,y)\omega(x)\,dx \quad\mbox{and}\quad G_k: y\mapsto \int_a^y k(y,x)\omega(x)\,dx.
$$
We denote by $\mathcal{K}_\omega$ the set of kernel-functions 
defined by
$$
\mathcal{K}_\omega=\left\{k: \Omega\to \mathbb{R}_+ \mbox{ such that } F_k\in L^\infty([a,b[;\mathbb{R}),\, 
G_k\in L^\infty(]a,b];\mathbb{R})\right\}.
$$
For $(k_1,k_2)\in \mathcal{K}_\omega\times \mathcal{K}_\omega$, let 
$$
\delta_{k_1,k_2}(x,y)=\int_y^x k_1(x,z)k_2(z,y)\omega(z)\,dz,\quad (x,y)\in \Omega.
$$
We introduce the binary relation $\mathcal{R}$ in $\mathcal{K}_\omega$  by
$$
(k_1,k_2)\in \mathcal{K}_\omega\times \mathcal{K}_\omega,\quad k_1\mathcal{R} k_2 \Longleftrightarrow 0<\delta_{k_1,k_2}(x,y)<\infty,\, (x,y)\in \Omega.
$$
\begin{proposition}\label{PR1}
Let $(k_1,k_2)\in \mathcal{K}_\omega\times \mathcal{K}_\omega$ be such that $k_1\mathcal{R} k_2$. Then 
$$
\delta_{k_1,k_2}\in \mathcal{K}_\omega.
$$
\end{proposition}

\begin{proof}
First, by the definition of the binary relation $\mathcal{R}$, the function
$$
\delta_{k_1,k_2}: \Omega\to \mathbb{R}_+
$$
is well-defined. Let $a\leq y< b$, using Fubini's theorem, we have
\begin{eqnarray*}
\left|F_{\delta_{k_1,k_2}}(y)\right|&=&F_{\delta_{k_1,k_2}}(y)\\
&=&\int_y^b \int_y^x k_1(x,z)k_2(z,y)\omega(z)\,dz \omega(x)\,dx\\
&=& \int_y^b k_2(z,y) \omega(z)  \left(\int_z^b   k_1(x,z)\omega(x) \,dx\right)         \,dz\\
&=& \int_y^b k_2(z,y) \omega(z) F_{k_1}(z)\,dz\\
&\leq & \|F_{k_1}\|_{L^\infty([a,b[;\mathbb{R})}  \int_y^b k_2(z,y)\omega(z)\,dz\\
&=&\|F_{k_1}\|_{L^\infty([a,b[;\mathbb{R})}   F_{k_2}(y)\\
&\leq & \|F_{k_1}\|_{L^\infty([a,b[;\mathbb{R})} \|F_{k_2}\|_{L^\infty([a,b[;\mathbb{R})},
\end{eqnarray*}
which proves that $F_{\delta_{k_1,k_2}}\in L^\infty([a,b[;\mathbb{R})$. Similarly, for $ a<y\leq b$, we obtain
$$
\left|G_{\delta_{k_1,k_2}}(y)\right|\leq \|G_{k_1}\|_{L^\infty(]a,b];\mathbb{R})} \|G_{k_2}\|_{L^\infty(]a,b];\mathbb{R})},
$$
which proves that $G_{\delta_{k_1,k_2}}\in L^\infty(]a,b];\mathbb{R})$. Therefore, $\delta_{k_1,k_2}$ belongs to $ \mathcal{K}_\omega$.
\end{proof}

Now, we introduce the notion of conjugate kernels, which will play a central role later.

\begin{definition}\label{dcc}
Let $(k,k')\in \mathcal{K}_\omega\times \mathcal{K}_\omega$. We say that $k$ and $k'$ are conjugate if
$$
\delta_{k,k'}(x,y)=\delta_{k',k}(x,y)=1,\quad (x,y)\in \Omega.
$$ 
\end{definition}

Given $k'\in \mathcal{K}_\omega$, let 
$$
\mbox{conj}(k')=\{k\in \mathcal{K}_\omega \mbox{ such that } k \mbox{ and }k' \mbox{ are conjugate}\}.
$$
Clearly, we have
$$
k\in \mbox{conj}(k') \Longleftrightarrow k'\in \mbox{conj}(k).
$$

Let us provide some examples of conjugate kernels.

\begin{example}\label{exconjRL}
Let $0<\alpha<1$ and  
$$
\omega(x)=1,\quad a\leq x\leq b.
$$
We define the kernel-functions $k_\alpha$ and $k'_\alpha$ by
$$
k_\alpha(x,y)=\frac{(x-y)^{\alpha-1} }{\Gamma(\alpha)},\quad (x,y)\in \Omega
$$
and 
$$
k'_\alpha(x,y)=\frac{(x-y)^{-\alpha} }{\Gamma(1-\alpha)},\quad (x,y)\in \Omega.
$$
Then $k_\alpha, k'_\alpha\in \mathcal{K}_\omega$. Moreover, $k_\alpha$ and $k'_\alpha$ are conjugate.
\end{example}

\begin{example}\label{exconjHD}
Let $(a,b)\in \mathbb{R}^2$ be such that $0<a<b$. Let $0<\alpha<1$ and
$$
\omega(x)=\frac{1}{x},\quad x\in [a,b].
$$
We define the kernel-functions $k_\alpha$ and $k'_\alpha$ by
$$
k_\alpha(x,y)=\frac{1}{\Gamma(\alpha)} \left(\ln \frac{x}{y}\right)^{\alpha-1},\quad (x,y)\in \Omega
$$
and
$$
k'_\alpha(x,y)=\frac{1}{\Gamma(1-\alpha)} \left(\ln \frac{x}{y}\right)^{-\alpha},\quad (x,y)\in \Omega.
$$
Then $k_\alpha, k'_\alpha\in \mathcal{K}_\omega$. Moreover, $k_\alpha$ and $k'_\alpha$ are conjugate.
\end{example}

\begin{example}\label{exconjER}
Let $(a,b)\in \mathbb{R}^2$ be such that $0<a<b$, $\sigma>0$ and $0<\alpha<1$. Let
$$
\omega(x)=\sigma x^{\sigma-1},\quad x\in [a,b].
$$
We define the kernel-functions $k_\alpha$ and $k'_\alpha$ by
$$
k_\alpha(x,y)=\frac{1}{\Gamma(\alpha)} \left(x^\sigma-y^\sigma\right)^{\alpha-1},\quad (x,y)\in \Omega
$$
and
$$
k'_\alpha(x,y)=\frac{1}{\Gamma(1-\alpha)} \left(x^\sigma-y^\sigma\right)^{-\alpha},\quad (x,y)\in \Omega.
$$
Then $k_\alpha, k'_\alpha\in \mathcal{K}_\omega$. Moreover, $k_\alpha$ and $k'_\alpha$ are conjugate.
\end{example}

\begin{example}[Sonine kernels]
Let   $(a,b)=(0,1)$ and
$$
\omega(x)=1,\quad 0\leq x\leq 1.
$$
Let
$$
k_1(x,y)=\mathbb{K}_1(x-y),\quad (x,y)\in \Omega,
$$
where the kernel $\mathbb{K}_1$ is assumed to satisfy the condition that there exists another kernel 
$$
k_2(x,y)=\mathbb{K}_2(x-y),\quad (x,y)\in \Omega,
$$
such that 
\begin{equation}\label{Sonine}
\int_0^t \mathbb{K}_1(t-s)\mathbb{K}_2(s)\,ds=1,\quad t>0.
\end{equation}
Then $k_1$ and $k_2$ are conjugate. Note that \eqref{Sonine} is known in the literature as Sonine condition, which is due  to Sonine \cite{Son}.
\end{example}

\section{Integral operators involving a kernel-function}\label{sec3}

Let $\omega\in \mathcal{W}$ be a fixed weight function and $k: \Omega\to \mathbb{R}_+$ be an arbitrary kernel-function that belongs to $\mathcal{K}_\omega$. We define the $k$-integrals of a function $f\in L^1([a,b];\mathbb{R})$ as follows.

\begin{definition}
Let $f\in L^1([a,b];\mathbb{R})$. The left-sided $k$-integral of $f$ is given by
$$
(I_{a}^kf)(x)=\int_a^x k(x,y)f(y)\omega(y)\,dy,\quad \mbox{ a.e. } x\in [a,b].
$$
The right-sided $k$-integral of $f$ is given by
$$
(I_{b}^kf)(x)=\int_x^b k(y,x)f(y)\omega(y)\,dy,\quad \mbox{ a.e. } x\in [a,b].
$$
\end{definition}
Next, we shall establish some properties related to the $k$-integral operators defined above.

\begin{proposition}\label{pr1}
The left-sided $k$-integral operator 
$$
I_{a}^k: L^1([a,b];\mathbb{R})\to L^1([a,b];\mathbb{R})
$$
is  linear and continuous. Moreover, for all $f\in L^1([a,b];\mathbb{R})$, we have
\begin{equation}\label{p1}
\|I_{a}^kf\|_{L^1([a,b];\mathbb{R})}\leq  \|\omega^{-1}\|_{L^\infty([a,b];\mathbb{R})} \|\omega\|_{L^\infty([a,b];\mathbb{R})}\|F_{k}\|_{L^\infty([a,b[;\mathbb{R})}\|f\|_{L^1([a,b];\mathbb{R})}.
\end{equation}
\end{proposition}

\begin{proof}
Let $f\in L^1([a,b];\mathbb{R})$. By Fubini's theorem, we have
\begin{eqnarray*}
\|I_{a}^kf\|_{L^1([a,b];\mathbb{R})}&=&\int_a^b |(I_{a}^kf)(x)|\,dx\\
&\leq & \int_a^b \int_a^x k(x,y)|f(y)|\omega(y)\,dy\,dx\\
&=& \int_a^b |f(y)| \omega(y)\left(\int_y^b k(x,y)\,dx\right)\,dy\\
&=& \int_a^b |f(y)| \omega(y)\left(\int_y^b \omega^{-1}(x)k(x,y)\omega(x)\,dx\right)\,dy\\
&\leq & \|\omega^{-1}\|_{L^\infty([a,b];\mathbb{R})} 
\int_a^b |f(y)|\omega(y) F_k(y)\,dy\\
&\leq &   \|\omega^{-1}\|_{L^\infty([a,b];\mathbb{R})} \|\omega\|_{L^\infty([a,b];\mathbb{R})}\|F_{k}\|_{L^\infty([a,b[;\mathbb{R})}\|f\|_{L^1([a,b];\mathbb{R})},
\end{eqnarray*}
which proves \eqref{p1}.
\end{proof}

Similarly, we have the following property for the right-sided $k$-integral operator.

\begin{proposition}\label{pr1b}
The right-sided $k$-integral operator 
$$
I_{b}^k: L^1([a,b];\mathbb{R})\to L^1([a,b];\mathbb{R})
$$
is  linear and continuous. Moreover, for all $f\in L^1([a,b];\mathbb{R})$, we have
$$
\|I_{b}^kf\|_{L^1([a,b];\mathbb{R})}\leq  \|\omega^{-1}\|_{L^\infty([a,b];\mathbb{R})} \|\omega\|_{L^\infty([a,b];\mathbb{R})}\|G_{k}\|_{L^\infty(]a,b];\mathbb{R})}\|f\|_{L^1([a,b];\mathbb{R})}.
$$
\end{proposition}
The proof of the above result is similar to that of Proposition \ref{PR1}, so we omit it.

Further,  we calculate $I_a^k f$ and $I_b^kf$ for some particular functions $f\in L^1([a,b];\mathbb{R})$. First, we consider the case $f\equiv C$, where $C\in \mathbb{R}$ is a certain constant. Then, we have
$$
(I_{a}^kf)(x)=C \int_a^x k(x,y)\omega(y)\,dy,\quad \mbox{ a.e. } x\in [a,b].
$$ 
Let us introduce the unit kernel-function $\mathds{1}:\Omega\to \mathbb{R}_+$  defined by
\begin{equation}\label{onef}
\mathds{1}(x,y)=1,\quad (x,y)\in \Omega.
\end{equation}
It can be  seen that $\mathds{1}\in \mathcal{K}_\omega$. 
On the other hand, we have
\begin{eqnarray*}
\int_a^x k(x,y)\omega(y)\,dy&=& \int_a^x k(x,z)\mathds{1}(z,a)\omega(z)\,dz\\
&=&\delta_{k,\mathds{1}}(x,a),\quad \mbox{ a.e. } x\in [a,b].
\end{eqnarray*}
Therefore, we have the following result.

\begin{proposition}\label{ex1} 
Let $f\equiv C$, where $C\in \mathbb{R}$ is a certain constant.  Then
$$
(I_{a}^kf)(x)= C \delta_{k,\mathds{1}}(x,a),\quad \mbox{ a.e. }x\in [a,b],
$$
where $\mathds{1}$ is the unit kernel-function defined by \eqref{onef}.
\end{proposition}
Similarly, we have
\begin{proposition}\label{ex1b} 
Let $f\equiv C$, where $C\in \mathbb{R}$ is a certain constant.  Then
$$
(I_{b}^kf)(x)= C \delta_{\mathds{1},k}(b,x),\quad \mbox{ a.e. }x\in [a,b].
$$
\end{proposition}

Next, we consider the case 
$$
f(x)=k'(x,a),\quad \mbox{ a.e. } x\in [a,b],
$$
where $k'\in \mathcal{K}_\omega$ is a certain kernel-function. Observe that 
\begin{eqnarray*}
\int_a^b |f(x)|\,dx &=& \int_a^b k'(x,a)\,dx\\
&=& \int_a^b k'(x,a)\omega(x) \omega^{-1}(x)\,dx\\
&\leq & \|\omega^{-1}\|_{L^\infty([a,b];\mathbb{R})} \int_a^b k'(x,a)\omega(x) \,dx\\
&=&\|\omega^{-1}\|_{L^\infty([a,b];\mathbb{R})} F_{k'}(a)\\
&\leq & \|\omega^{-1}\|_{L^\infty([a,b];\mathbb{R})} \|F_{k'}\|_{L^\infty([a,b[;\mathbb{R})},
\end{eqnarray*}
which proves that $f\in L^1([a,b];\mathbb{R})$. On the other hand, for a.e. $x\in [a,b]$, we have
\begin{eqnarray*}
(I_a^{k}f)(x)&=& \int_a^x k(x,z) f(z)\omega(z)\,dz\\
&=& \int_a^x k(x,z) k'(z,a)\omega(z)\,dz\\
&=&\delta_{k,k'}(x,a).
\end{eqnarray*}
Therefore, we have the following result.
\begin{proposition}\label{ex2}
Let $f: [a,b]\to \mathbb{R}$ be the function given by
$$
f(x)=k'(x,a),\quad \mbox{ a.e. } x\in [a,b],
$$
where $k'\in \mathcal{K}_\omega$.  Then $f\in L^1([a,b];\mathbb{R})$. Moreover, we have 
$$
(I_a^{k}f)(x)=\delta_{k,k'}(x,a),\quad \mbox{ a.e. } x\in [a,b].
$$
\end{proposition}
Similarly, we have
\begin{proposition}\label{ex2b}
Let $f: [a,b]\to \mathbb{R}$ be the function given by
$$
f(x)=k'(b,x),\quad \mbox{ a.e. } x\in [a,b],
$$
where $k'\in \mathcal{K}_\omega$.  Then $f\in L^1([a,b];\mathbb{R})$. Moreover, we have 
$$
(I_b^{k}f)(x)=\delta_{k',k}(b,x),\quad \mbox{ a.e. } x\in [a,b].
$$
\end{proposition}

Further, we study the composition between   $k_i$-integral operators, $i=1,2$, where $(k_1,k_2)\in \mathcal{K}_\omega\times \mathcal{K}_\omega$. We have the following result.

\begin{theorem}\label{pr2}
Let $(k_1,k_2)\in \mathcal{K}_\omega\times \mathcal{K}_\omega$ be such that $k_1\mathcal{R}k_2$. Given a function $f\in L^1([a,b];\mathbb{R})$, we have
\begin{equation}\label{p2}
I_a^{k_1}\left(I_a^{k_2} f\right)(x)=\left(I_a^{k_3} f\right)(x),\quad \mbox{ a.e. }x\in [a,b],
\end{equation}
where $k_3=\delta_{k_1,k_2}$.
\end{theorem}

\begin{proof}
First, from Proposition \ref{pr1}, the mapping 
$$
I_a^{k_1}\circ I_a^{k_2}: L^1([a,b];\mathbb{R})\to L^1([a,b];\mathbb{R})
$$
is well-defined. Moreover, by Proposition \ref{PR1}, since  $k_1\mathcal{R}k_2$, we have $\delta_{k_1,k_2}\in \mathcal{K}_\omega$. Next, given $f\in L^1([a,b];\mathbb{R})$, 
for a.e. $x\in [a,b]$, using Fubini's theorem, we have
\begin{eqnarray*}
I_a^{k_1}\left(I_a^{k_2} f\right)(x)&=& \int_a^x k_1(x,y)\left(I_a^{k_2} f\right)(y)\omega(y)\,dy\\
&=& \int_a^x k_1(x,y) \int_a^y k_2(y,z)f(z)\omega(z)\,dz\,\omega(y)dy\\
&=& \int_a^x \left(\int_z^x  k_1(x,y)  k_2(y,z) \omega(y)  \,dy\right) f(z) \omega(z)\,dz\\
&=&\int_a^x \delta_{k_1,k_2}(x,z)f(z)\omega(z)\,dz\\
&=&\left(I_a^{\delta_{k_1,k_2}} f\right)(x),
\end{eqnarray*}
which proves \eqref{p2}.
\end{proof}
Using a similar argument as above, we obtain the following composition result for right-sided $k_i$-integral operators.

\begin{theorem}\label{pr2b}
Let $(k_1,k_2)\in \mathcal{K}_\omega\times \mathcal{K}_\omega$ be such that $k_2\mathcal{R}k_1$. Given a function $f\in L^1([a,b];\mathbb{R})$, we have
$$
I_b^{k_1}\left(I_b^{k_2} f\right)(x)=\left(I_b^{k_3} f\right)(x),\quad \mbox{ a.e. }x\in [a,b],
$$
where $k_3=\delta_{k_2,k_1}$.
\end{theorem}

\begin{remark}
Observe that in the particular case when 
$$
\omega\equiv 1,\,\, k_1(x,y)=\mathbb{K}_1(x-y),\,\,
k_2(x,y)=\mathbb{K}_2(x-y),\quad (x,y)\in \Omega,
$$
we have
\begin{equation}\label{TRV1}
\left(I_a^{k_i}f\right)(x)=\left(\mathbb{K}_i*f\right)(x),\quad \mbox{a.e. } x\in [a,b],\,\, i=1,2,
\end{equation}
where $*$ denotes the convolution product. Moreover, for $(x,y)\in \Omega$, using the change of variable $\tau=z-y$, we have
\begin{eqnarray*}
k_3(x,y)&:=&\delta_{k_1,k_2}(x,y)\\
&=& \int_y^x k_1(x,z)k_2(z,y)\,dz\\
&=&\int_y^x \mathbb{K}_1(x-z)\mathbb{K}_2(z-y)\,dz\\
&=&\int_0^{x-y} \mathbb{K}_1((x-y)-\tau)\mathbb{K}_2(\tau)\,d\tau\\
&=& \left(\mathbb{K}_1*\mathbb{K}_2\right)(x-y).
\end{eqnarray*}
Therefore, 
\begin{equation}\label{TRV2}
\left(I_a^{k_3}f\right)(x)=\left(\left(\mathbb{K}_1*\mathbb{K}_2\right)*f\right)(x),\quad \mbox{a.e. } x\in [a,b].
\end{equation}
Hence, using \eqref{TRV1}, \eqref{TRV2} and  Theorem \ref{pr2}, we obtain
$$
\left(\mathbb{K}_1*\left(\mathbb{K}_2*f\right)\right)(x)=\left(\left(\mathbb{K}_1*\mathbb{K}_2\right)*f\right)(x),\quad \mbox{a.e. } x\in [a,b],
$$
which is the standard  associativity property for  convolution products.

For general kernels $(k_1,k_2)\in \mathcal{K}_\omega\times \mathcal{K}_\omega$,  Theorem \ref{pr2} can be considered as a generalization of the  associativity property for convolution products. 
\end{remark}

\begin{remark}
Observe that if $\omega\equiv 1$, $k_1$ and $k_2$ are conjugate (see Definition \ref{dcc}) and $k_3=\delta_{k_1,k_2}$, then 
$$
\left(I_a^{k_3}f\right)(x)= \int_a^x f(y)\,dy,\quad \mbox{a.e. } x\in [a,b]
$$
and 
$$
\left(I_b^{k_3}f\right)(x)= \int_x^b f(y)\,dy,\quad \mbox{a.e. } x\in [a,b].
$$
Therefore, by Theorems \ref{pr2} and \ref{pr2b}, we obtain
$$
I_a^{k_1}\left(I_a^{k_2} f\right)(x)=\int_a^x f(y)\,dy,\quad \mbox{ a.e. }x\in [a,b]
$$
and
$$
I_b^{k_1}\left(I_b^{k_2} f\right)(x)= \int_x^b f(y)\,dy,\quad \mbox{ a.e. }x\in [a,b].
$$
\end{remark}

Next, we shall establish the following integration by parts rule.

\begin{theorem}\label{INTPA}
Let $f\in L^1([a,b];\mathbb{R})$ and $g\in L^\infty([a,b];\mathbb{R})$.  Then
\begin{equation}\label{RL1}
\int_a^b (I_a^kf)(x) g(x) \omega(x)\,dx=\int_a^b I_b^k(g)(x) f(x)\omega(x)\,dx.
\end{equation}
\end{theorem}

\begin{proof}
First, since $f\in L^1([a,b];\mathbb{R})$ and $g\omega \in L^\infty([a,b];\mathbb{R})$, we have $(I_a^kf)g\omega \in L^1([a,b];\mathbb{R})$. Next, using Fubini's theorem, we obtain
\begin{eqnarray*}
\int_a^b (I_a^kf)(x) g(x)\omega(x)\,dx&=&\int_a^b  \int_a^x k(x,y) f(y)\omega(y)\,dy g(x)\omega(x)\,dx\\
&=&\int_a^b \left(\int_y^b k(x,y)g(x)\omega(x)\,dx\right) f(y)\omega(y)\,dy\\
&=& \int_a^b I_b^k(g)(y) f(y)\omega(y)\,dy,
\end{eqnarray*}
which proves \eqref{RL1}.
\end{proof}

\begin{example}[Riemann-Liouville fractional integrals]\label{exRLI}
Let $\alpha>0$ and
$$
\omega(x)=1,\quad a\leq x\leq b.
$$
Let $k_\alpha$ be the Riemann-Liouville fractional kernel-function given by
$$
k_\alpha(x,y)=\frac{(x-y)^{\alpha-1}}{\Gamma(\alpha)}, \quad (x,y)\in \Omega.
$$
Then $I_a^{k_\alpha}$ is the left-sided Riemann-Liouville fractional integral of order $\alpha$, and  $I_b^{k_\alpha}$ is the right-sided Riemann-Liouville fractional integral of order $\alpha$ (see \cite{KST,SKM}).
\end{example}

\begin{example}[Hadamard fractional integrals]\label{exHDI}
Let $(a,b)\in \mathbb{R}^2$ be such that $0<a<b$. Let $\alpha>0$ and 
$$
\omega(x)=\frac{1}{x},\quad x\in [a,b].
$$
Let $k_\alpha$  be the Hadamard fractional kernel-function given  by
$$
k_\alpha(x,y)=\frac{1}{\Gamma(\alpha)} \left(\ln \frac{x}{y}\right)^{\alpha-1},\quad (x,y)\in \Omega.
$$
Then $I_a^{k_\alpha}$ is the left-sided Hadamard fractional integral of order $\alpha$, and  $I_b^{k_\alpha}$ is the right-sided Hadamard fractional integral of order $\alpha$ (see \cite{KST,SKM}).
\end{example}

\begin{remark}
Note that the most well-known properties related to Riemann-Liouville fractional integrals and Hadamard  fractional integrals can be deduced  from the obtained results in this section.
\end{remark}

\section{Derivative operators with respect to a kernel-function}\label{sec4}

In this section, we introduce derivative operators with respect to a kernel-function. First, let us recall briefly some notions that will be used later. 
\begin{definition}
Let $\varphi\in L^1([a,b];\mathbb{R})$. We say that $\varphi$ has  an absolutely continuous representative, if there exists a function $\psi\in AC([a,b];\mathbb{R})$, where $AC([a,b];\mathbb{R})$ denotes the space of real valued and absolutely continuous functions in $[a,b]$, such that 
$$
\varphi(x)=\psi(x),\quad \mbox{ a.e. } x\in [a,b].
$$
In this case, the function $\varphi$ is identified to its absolutely continuous representative $\psi$.
\end{definition} 
Recall that a function $\psi$ belongs to the functional space 
$AC([a,b];\mathbb{R})$ if and only if there exist a constant $c\in \mathbb{R}$ and a function $\mu \in L^1([a,b];\mathbb{R})$ such that
$$
\psi(x)=c+\int_a^x \mu(t)\,dt,\quad x\in [a,b].
$$
In this case, we have
$$
c=\psi(a)
$$
and
$$
\frac{d\psi}{dx}(x)=\mu(x),\quad \mbox{ a.e. } x\in [a,b].
$$
For more details on such spaces, see for example \cite{KST,SKM}.

Next, let us fix $\omega\in \mathcal{W}$ and $k'\in \mathcal{K}_\omega$.

\begin{definition}
Let $f\in L^1([a,b];\mathbb{R})$ be such that $I_a^{k'}f$ has an absolutely continuous representative.  We define the left-sided $k'$-derivative of $f$  by
$$
(D_a^{k'}f)(x)= \left(\frac{1}{\omega(x)}\frac{d}{dx}\right) (I_a^{k'} f)(x),\quad \mbox{ a.e. } x\in [a,b].
$$
Let $f\in L^1([a,b];\mathbb{R})$ be such that $I_b^{k'}f$ has an absolutely continuous representative.  We define the right-sided $k'$-derivative of $f$  by
$$
(D_b^{k'}f)(x)= -\left(\frac{1}{\omega(x)}\frac{d}{dx}\right) (I_b^{k'} f)(x),\quad \mbox{ a.e. } x\in [a,b].
$$
\end{definition}

Further, let us give some examples, where we compute the $k'$-derivatives of some particular functions. 

First, we consider the case $f\equiv C$, where $C\in \mathbb{R}$ is a certain constant. By Proposition \ref{ex1}, we know that 
$$
(I_a^{k'}f)(x)=C \delta_{k',\mathds{1}}(x,a),\quad \mbox{ a.e. }x\in [a,b],
$$
i.e., 
$$
I_a^{k'}f= C \delta_{k',\mathds{1}}(\cdot,a)
$$
in $L^1([a,b];\mathbb{R})$.  Hence, if $\delta_{k',\mathds{1}}(\cdot,a)$ has an absolutely continuous representative, then
$$
(D_a^{k'}f)(x)= C \left(\frac{1}{\omega(x)}\frac{d}{dx}\right) \delta_{k',\mathds{1}}(x,a),\quad \mbox{ a.e. }x\in [a,b].
$$
Therefore, we have the following result.

\begin{proposition}\label{ex21}
Suppose that $\delta_{k',\mathds{1}}(\cdot,a)$ has an absolutely continuous representative. Let $f\equiv C$, where $C\in \mathbb{R}$ is a certain constant. Then
$$
(D_a^{k'}f)(x)= C \left(\frac{1}{\omega(x)}\frac{d}{dx}\right) \delta_{k',\mathds{1}}(x,a),\quad \mbox{ a.e. }x\in [a,b].
$$
\end{proposition}

Similarly, using Proposition \ref{ex1b}, we deduce the following result.

\begin{proposition}\label{ex21b}
Suppose that $\delta_{\mathds{1},k'}(b,\cdot)$ has an absolutely continuous representative. Let $f\equiv C$, where $C\in \mathbb{R}$ is a certain constant. Then
$$
(D_b^{k'}f)(x)= -C \left(\frac{1}{\omega(x)}\frac{d}{dx}\right) \delta_{\mathds{1},k'}(b,x),\quad \mbox{ a.e. }x\in [a,b].
$$
\end{proposition}

Next, we consider the case when
$$
f(x)=k''(x,a),\quad \mbox{ a.e. } x\in [a,b],
$$
where $k''\in \mathcal{K}_\omega$ is a certain kernel-function. By Proposition \ref{ex2}, we know that 
$$
(I_a^{k'}f)(x)=\delta_{k',k''}(x,a),\quad \mbox{ a.e. } x\in [a,b].
$$
Therefore, if $\delta_{k',k''}(\cdot,a)$ has an absolutely continuous representative, then
$$
(D_{a}^{k'}f)(x)=\left(\frac{1}{\omega(x)}\frac{d}{dx}\right)\delta_{k',k''}(x,a),\quad \mbox{ a.e. } x\in [a,b].
$$ 
Hence, we obtain the following result.

\begin{proposition}\label{pd2k}
Let $k''\in \mathcal{K}_\omega$. Let $f: [a,b]\to \mathbb{R}$ be the function given by
$$
f(x)=k''(x,a), \quad \mbox{ a.e. } x\in [a,b].
$$ 
If $\delta_{k',k''}(\cdot,a)$ has an absolutely continuous representative, then
$$
(D_{a}^{k'}f)(x)=\left(\frac{1}{\omega(x)}\frac{d}{dx}\right)\delta_{k',k''}(x,a),\quad \mbox{ a.e. } x\in [a,b].
$$ 
\end{proposition}
Similarly, using Proposition \ref{ex2b}, we obtain the following result.
\begin{proposition}\label{pd2kb}
Let $k''\in \mathcal{K}_\omega$. Let $f: [a,b]\to \mathbb{R}$ be the function given by
$$
f(x)=k''(b,x), \quad \mbox{ a.e. } x\in [a,b].
$$ 
If $\delta_{k'',k'}(b,\cdot)$ has an absolutely continuous representative, then
$$
(D_{b}^{k'}f)(x)=-\left(\frac{1}{\omega(x)}\frac{d}{dx}\right)\delta_{k'',k'}(b,x),\quad \mbox{ a.e. } x\in [a,b].
$$ 
\end{proposition}

Now, we discuss the relation between $k$-integral and $k'$-derivative operators, where $k\in \mbox{conj}(k')$. In this part, the reader will be aware of the importance of the conjugate kernels concept.

Let $k'\in \mathcal{K}_\omega$ and
$$
k\in \mbox{conj}(k').
$$

\begin{theorem}\label{TN1}
For every $f\in L^1([a,b];\mathbb{R})$, we have
\begin{equation}\label{inv1}
D_a^{k'}(I_a^kf)(x)= f(x),\quad \mbox{ a.e. }x\in [a,b]
\end{equation}
and
\begin{equation}\label{inv1b}
D_b^{k'}(I_b^kf)(x)= f(x),\quad \mbox{ a.e. }x\in [a,b].
\end{equation}
\end{theorem}

\begin{proof}
Let $f\in L^1([a,b];\mathbb{R})$. First, using Theorem \ref{pr2}, we have
$$
I_a^{k'}\circ I_a^kf=I_a^{\delta_{k',k}}f.
$$
Since $k$ and $k'$ are  conjugate, we have 
$$
\delta_{k',k}(x,y)=1,\quad (x,y)\in \Omega.
$$
Therefore, 
$$
I_a^{k'}(I_a^kf)(x)=\int_a^x f(y)\omega(y)\,dy,\quad \mbox{ a.e. } x\in [a,b].
$$
Observe that the function
$$
x\mapsto \int_a^x f(y)\omega(y)\,dy
$$
is an absolutely continuous representative of $I_a^{k'}(I_a^kf)$. Hence, for a.e. $x\in [a,b]$, we obtain
\begin{eqnarray*}
D_a^{k'}(I_a^kf)(x)&=&\left(\frac{1}{\omega(x)}\frac{d}{dx}\right) \int_a^x f(y)\omega(y)\,dy\\
&=&\frac{1}{\omega(x)} f(x) \omega(x)\\
&=& f(x),
\end{eqnarray*}
which proves \eqref{inv1}. Similarly, using Theorem \ref{pr2b} and the fact that $k$ and $k'$ are conjugate, we obtain
$$
I_b^{k'}(I_b^kf)(x)=\int_x^b f(y)\omega(y)\,dy,\quad \mbox{ a.e. } x\in [a,b],
$$
i.e.,
$$
I_b^{k'}(I_b^kf)(x)=\int_a^b f(y)\omega(y)\,dy+\int_a^x -f(y)\omega(y)\,dy,\quad \mbox{ a.e. } x\in [a,b].
$$
Therefore, the function
$$
x\mapsto \int_x^b f(y)\omega(y)\,dy
$$
is an absolutely continuous representative of $I_b^{k'}(I_b^kf)$. Hence, for a.e. $x\in [a,b]$, we obtain
\begin{eqnarray*}
D_b^{k'}(I_b^kf)(x)&=&-\left(\frac{1}{\omega(x)}\frac{d}{dx}\right) \int_x^b f(y)\omega(y)\,dy\\
&=&\frac{1}{\omega(x)} f(x) \omega(x)\\
&=& f(x),
\end{eqnarray*}
which proves \eqref{inv1b}.
\end{proof}

\begin{remark}
In \cite{KT}, Kochubei introduced a certain    class of integro-differential operators involving a kernel function 
$$
k(t,s)=\mathbb{K}(t-s).
$$
In order to guaranty that the operator possesses a right inverse,  Sonine condition is imposed. 
For more general kernels, we use the conjugate kernels condition (see  Theorem \ref{TN1}).  \end{remark}

Further,  we denote by $I_a^k(L^1)$ the set of functions $f\in L^1([a,b];\mathbb{R})$ such that there exists $\varphi\in L^1([a,b];\mathbb{R})$ satisfying
$$
f(x)=(I_a^k \varphi)(x),\quad \mbox{ a.e. } x\in [a,b].
$$
Similarly, we denote by $I_b^k(L^1)$ the set of functions $f\in L^1([a,b];\mathbb{R})$ such that there exists $\varphi\in L^1([a,b];\mathbb{R})$ satisfying
$$
f(x)=(I_b^k \varphi)(x),\quad \mbox{ a.e. } x\in [a,b].
$$

\begin{theorem}\label{TN2}
For every $f\in I_a^k(L^1)$, we have
\begin{equation}\label{inv2}
I_a^k \left(D_a^{k'}f\right)(x)= f(x),\quad \mbox{ a.e. }x\in [a,b].
\end{equation}
For every $f\in I_b^k(L^1)$, we have
\begin{equation}\label{inv2b}
I_b^k \left(D_b^{k'}f\right)(x)= f(x),\quad \mbox{ a.e. }x\in [a,b].
\end{equation}
\end{theorem}

\begin{proof}
Let $f\in I_a^k(L^1)$. Then there exists some $\varphi\in L^1([a,b];\mathbb{R})$ such that 
\begin{equation}\label{JAL1}
f(x)=(I_a^k \varphi)(x),\quad \mbox{ a.e. } x\in [a,b].
\end{equation}
Using Theorem \ref{pr2}, we obtain
$$
I_a^{k'}f(x) =\left(I_a^{\delta_{k',k}}\varphi\right), \quad \mbox{ a.e. } x\in [a,b].
$$
Since $k$ and $k'$ are  conjugate, we have 
$$
\delta_{k',k}(x,y)=1,\quad (x,y)\in \Omega,
$$
which yields
$$
I_a^{k'}f(x) =\int_a^x \varphi(y)\omega(y)\,dy,\quad \mbox{ a.e. } x\in [a,b].
$$
Hence,
\begin{eqnarray*}
(D_a^{k'}f)(x)&=&\left(\frac{1}{\omega(x)}\frac{d}{dx}\right) (I_a^{k'}f)(x)\\
&=& \left(\frac{1}{\omega(x)}\frac{d}{dx}\right) \int_a^x \varphi(y)\omega(y)\,dy\\
&=&\frac{1}{\omega(x)} \varphi(x) \omega(x)\\
&=& \varphi(x),\quad \mbox{ a.e. } x\in [a,b],
\end{eqnarray*}
which implies that 
\begin{equation}\label{JAL2}
I_a^k \left(D_a^{k'}f\right)(x)=(I_a^k \varphi)(x),\quad \mbox{ a.e. } x\in [a,b].
\end{equation}
Combining \eqref{JAL1} with \eqref{JAL2}, \eqref{inv2} follows. Next, let $f\in I_b^k(L^1)$. Then there exists some $\varphi\in L^1([a,b];\mathbb{R})$ such that 
\begin{equation}\label{JAL1b}
f(x)=(I_b^k \varphi)(x),\quad \mbox{ a.e. } x\in [a,b].
\end{equation}
Using Theorem \ref{pr2b}, we obtain
$$
I_b^{k'}f(x) =\left(I_b^{\delta_{k,k'}}\varphi\right), \quad \mbox{ a.e. } x\in [a,b],
$$
which yields
$$
I_b^{k'}f(x) =\int_x^b \varphi(y)\omega(y)\,dy,\quad \mbox{ a.e. } x\in [a,b].
$$
Hence,
\begin{eqnarray*}
(D_b^{k'}f)(x)&=&-\left(\frac{1}{\omega(x)}\frac{d}{dx}\right) (I_b^{k'}f)(x)\\
&=& -\left(\frac{1}{\omega(x)}\frac{d}{dx}\right) \int_x^b \varphi(y)\omega(y)\,dy\\
&=&\frac{1}{\omega(x)} \varphi(x) \omega(x)\\
&=& \varphi(x),\quad \mbox{ a.e. } x\in [a,b],
\end{eqnarray*}
which implies that 
\begin{equation}\label{JAL2b}
I_b^k \left(D_b^{k'}f\right)(x)=(I_b^k \varphi)(x),\quad \mbox{ a.e. } x\in [a,b].
\end{equation}
Combining \eqref{JAL1b} with \eqref{JAL2b}, \eqref{inv2b} follows.
\end{proof}

Note that in general, \eqref{inv2} and \eqref{inv2b} are not true for any $f\in L^1([a,b];\mathbb{R})$. The following results show this fact.

\begin{theorem}\label{TN3}
Let $f\in L^1([a,b];\mathbb{R})$ be such that $I_a^{k'}f$
has an absolutely continuous representative. Then
\begin{equation}\label{inv3}
I_a^k \left(D_a^{k'}f\right)(x)=f(x)-(I_a^{k'}f)(a) \left(\frac{1}{\omega(x)}\frac{d}{dx}\right)\delta_{k,\mathds{1}}(x,a),\quad \mbox{ a.e. } x\in [a,b].
\end{equation}
\end{theorem}

\begin{proof}
Since $I_a^{k'}f$ has an absolutely continuous representative, then it can be identified to an absolutely
continuous function in $[a,b]$. Therefore, there exists some $\varphi\in L^1([a,b];\mathbb{R})$ such that
$$
(I_a^{k'}f)(x)=(I_a^{k'}f)(a)+\int_a^x \varphi(z)\,dz,\quad 
a\leq x\leq b.
$$
Therefore, 
$$
\frac{d}{dx} (I_a^{k'}f)(x)=\varphi(x),\quad \mbox{ a.e. } x\in [a,b],
$$
which yields
\begin{equation}\label{nsitt}
(D_a^{k'}f)(x)=\varphi(x)\omega^{-1}(x),\quad \mbox{ a.e. } x\in [a,b].
\end{equation}
On the other hand, since $k$ and $k'$ are conjugate, we have
\begin{eqnarray*}
(I_a^{k'}f)(x)&=&(I_a^{k'}f)(a)+\int_a^x \omega^{-1}(z) \varphi(z)\omega(z)\,dz\\
&=& (I_a^{k'}f)(a)+ (I_a^{\mathds{1}}\omega^{-1}\varphi)(x)\\
&=& (I_a^{k'}f)(a)+\left(I_a^{\delta_{k',k}}\omega^{-1}\varphi\right)(x),\quad  \mbox{ a.e. } x\in [a,b].
\end{eqnarray*}
By Theorem \ref{pr2}, we have
$$
\left(I_a^{\delta_{k',k}}\omega^{-1}\varphi\right)(x)=I_a^{k'}\left(I_a^k \omega^{-1}\varphi\right)(x),
\quad \mbox{ a.e. } x\in [a,b].
$$
Therefore,
$$
(I_a^{k'}f)(x)=(I_a^{k'}f)(a)+ I_a^{k'}\left(I_a^k\omega^{-1}\varphi\right)(x),\quad \mbox{ a.e. } x\in [a,b],
$$
which yields
$$
I_a^{k'}\left(f-I_a^k\omega^{-1}\varphi\right)(x)=(I_a^{k'}f)(a),\quad \mbox{ a.e. } x\in [a,b].
$$
Then, we obtain
\begin{equation}\label{gjl}
I_a^k \left(I_a^{k'}\left(f-I_a^k\omega^{-1}\varphi\right)\right)(x)=I_a^k\left((I_a^{k'}f)(a)\right)(x),\quad \mbox{ a.e. } x\in [a,b].
\end{equation}
Using Proposition \ref{ex1}, we obtain
\begin{equation}\label{hmd}
I_a^k\left((I_a^{k'}f)(a)\right)(x)=(I_a^{k'}f)(a) \delta_{k,\mathds{1}}(x,a),\quad \mbox{ a.e. } x\in [a,b].
\end{equation}
Again, using Theorem \ref{pr2} and the fact that $k$ and $k'$ are conjugate,  we get
$$
I_a^k \left(I_a^{k'}\left(f-I_a^k\omega^{-1}\varphi\right)\right)(x)=
I_a^{\mathds{1}} \left(f-I_a^k\omega^{-1}\varphi\right)(x),\quad \mbox{ a.e. } x\in [a,b],
$$
i.e.,
\begin{equation}\label{hmD}
I_a^k \left(I_a^{k'}\left(f-I_a^k\omega^{-1}\varphi\right)\right)(x)= \int_a^x \left(f-I_a^k\omega^{-1}\varphi\right)(y)\omega(y)\,dy, \quad \mbox{ a.e. } x\in [a,b].
\end{equation}
Using \eqref{gjl}, \eqref{hmd} and \eqref{hmD}, we obtain
$$
\int_a^x \left(f-I_a^k\omega^{-1}\varphi\right)(y)\omega(y)\,dy=(I_a^{k'}f)(a) \delta_{k,\mathds{1}}(x,a),\quad \mbox{ a.e. } x\in [a,b],
$$
which yields
$$
\left(f-I_a^k\omega^{-1}\varphi\right)(x)=(I_a^{k'}f)(a) \left(\frac{1}{\omega(x)}\frac{d}{dx}\right) \delta_{k,\mathds{1}}(x,a),\quad \mbox{ a.e. } x\in [a,b],
$$
i.e.,
$$
(I_a^k\omega^{-1}\varphi)(x)=f(x)-(I_a^{k'}f)(a) \left(\frac{1}{\omega(x)}\frac{d}{dx}\right) \delta_{k,\mathds{1}}(x,a),\quad \mbox{ a.e. } x\in [a,b].
$$
Finally, by \eqref{nsitt}, we obtain
$$
I_a^k \left(D_a^{k'}f\right)(x)=f(x)-(I_a^{k'}f)(a) \left(\frac{1}{\omega(x)}\frac{d}{dx}\right) \delta_{k,\mathds{1}}(x,a),\quad \mbox{ a.e. } x\in [a,b],
$$
which proves \eqref{inv3}.
\end{proof}

Using a similar argument as above, we obtain the following result.

\begin{theorem}\label{TN3b}
Let $f\in L^1([a,b];\mathbb{R})$ be such that $I_b^{k'}f$
has an absolutely continuous representative. Then
$$
I_b^k \left(D_b^{k'}f\right)(x)=f(x)+(I_b^{k'}f)(b) \left(\frac{1}{\omega(x)}\frac{d}{dx}\right) \delta_{\mathds{1},k}(b,x),\quad \mbox{ a.e. } x\in [a,b].
$$
\end{theorem}

We end this section with some standard examples of $k'$-derivatives.

\begin{example}\label{exRLD}
Let $0<\alpha<1$ and  
$$
\omega(x)=1,\quad a\leq x\leq b.
$$
We define the kernel-function $k'_\alpha$ by
$$
k'_\alpha(x,y)=\frac{(x-y)^{-\alpha} }{\Gamma(1-\alpha)},\quad (x,y)\in \Omega.
$$
We mentioned in Example \ref{exconjRL} that 
$$
k_\alpha(x,y)=\frac{(x-y)^{\alpha-1} }{\Gamma(\alpha)},\quad (x,y)\in \Omega,
$$
belongs to $\mbox{conj}(k'_\alpha)$, which means that 
$$
\mbox{conj}(k'_\alpha)\neq\emptyset.
$$
It can be  seen (see \cite{KST,SKM}) that 
$D_a^{k'_\alpha}$ is the left-sided Riemann-Liouville fractional derivative of order $\alpha$ and 
$D_b^{k'_\alpha}$ is the right-sided Riemann-Liouville fractional derivative of order $\alpha$. Therefore, the most well-known properties related to Riemann-Liouville fractional derivatives can be deduced  from the obtained results in this section.
\end{example}

\begin{example}\label{exHD}
Let $(a,b)\in \mathbb{R}^2$ be such that $0<a<b$. Let $0<\alpha<1$ and
$$
\omega(x)=\frac{1}{x},\quad x\in [a,b].
$$
We define the kernel-function $k'_\alpha$ by
$$
k'_\alpha(x,y)=\frac{1}{\Gamma(1-\alpha)} \left(\ln \frac{x}{y}\right)^{-\alpha},\quad (x,y)\in \Omega.
$$
We mentioned in Example \ref{exconjHD} that
$$
k_\alpha(x,y)=\frac{1}{\Gamma(\alpha)} \left(\ln \frac{x}{y}\right)^{\alpha-1},\quad (x,y)\in \Omega
$$
belongs to $\mbox{conj}(k'_\alpha)$, which means that 
$$
\mbox{conj}(k'_\alpha)\neq\emptyset.
$$
It can be seen (see \cite{KST,SKM}) that 
$D_a^{k'_\alpha}$ is the left-sided Hadamard fractional derivative of order $\alpha$ and 
$D_b^{k'_\alpha}$ is the right-sided Hadamard fractional derivative of order $\alpha$. Therefore, the most well-known properties related to Hadamard fractional derivatives can be deduced  from the obtained results in this section.
\end{example}

\section{New fractional operators}\label{sec5}

In this section, using the general approach presented previously, new fractional operators, which are different to those existing in the literature,  are introduced. Moreover, several interesting properties related to this kind of operators  are established. First, we need the following definitions and properties (see, for example \cite{AS}).

The exponential integral function is defined as:
$$
E_1(x)=\int_x^\infty \frac{e^{-t}}{t}\,dt,\quad  x>0.
$$
The lower incomplete gamma function is defined as:
$$
\gamma(s,x)=\int_0^x t^{s-1}e^{-t} \,dt,\quad s>0,\, x>0.
$$
The regularized lower Gamma function is given by 
$$
P(s,x)=\frac{\gamma(s,x)}{\Gamma(s)}, \quad s>0,\, x>0.
$$
The derivative of $P(s,x)$ with respect to $x$ is given by
\begin{equation}\label{dig}
\frac{d}{dx}P(s,x)=\frac{x^{s-1} e^{-x}}{\Gamma(s)},\quad s>0,\, x>0.
\end{equation}

Next, we take $[a,b]=[0,1]$ and 
$$
\omega(x)=1,\quad x\in [0,1].
$$
In this case, we have
$$
\Omega=\{(x,y)\in [0,1]\times [0,1]:\, x>y\}.
$$
For $\alpha>0$, we define the  kernel-functions
$$
k_\alpha(x,y)=\left(\int_0^\infty \frac{\left(\frac{x-y}{\alpha}\right)^{s-1}}{\Gamma(s)}\,ds\right)e^{-\left(\frac{x-y}{\alpha}\right)},\quad (x,y)\in \Omega
$$
and
$$
k'_\alpha(x,y)=\frac{1}{\alpha}E_1\left(\frac{x-y}{\alpha}\right),\quad (x,y)\in \Omega.
$$

\begin{proposition}\label{kk'}
For all $\alpha>0$, we have
$$
(k_\alpha,k'_\alpha)\in \mathcal{K}_\omega\times \mathcal{K}_\omega.
$$
\end{proposition}

\begin{proof}
Let us fix $\alpha>0$.  Given $0\leq y<1$, we have
\begin{eqnarray*}
|F_{k_\alpha}(y)| &=& \int_y^1 k_\alpha(x,y)\,dx \\
&=& \int_y^1 \left(\int_0^\infty \frac{\left(\frac{x-y}{\alpha}\right)^{s-1}}{\Gamma(s)}\,ds\right)e^{-\left(\frac{x-y}{\alpha}\right)}\,dx\\
&=& \alpha \int_0^{\frac{1-y}{\alpha}} \left(\int_0^\infty \frac{z^{s-1}}{\Gamma(s)}\,ds\right)e^{-z}\,dz\\
&=& \alpha \int_0^\infty \left(\int_0^{\frac{1-y}{\alpha}} z^{s-1} e^{-z}\,dz\right) \frac{1}{\Gamma(s)}\,ds\\
&=&\alpha \int_0^\infty \frac{\gamma\left(s,\frac{1-y}{\alpha}\right)}{\Gamma(s)}\,ds\\
&=&\alpha \int_0^\infty P\left(s,\frac{1-y}{\alpha}\right)\,ds.
\end{eqnarray*}
On the other hand, by \eqref{dig}, we know that $P(s,\cdot)$ is non-decreasing in $\mathbb{R}_+$. Therefore, we get
$$
|F_{k_\alpha}(y)|\leq \alpha \int_0^\infty P\left(s,\frac{1}{\alpha}\right)\,ds,\quad 0\leq y<1,
$$
which yields $F_{k_\alpha}\in L^\infty([0,1[;\mathbb{R})$. Using a similar calculation, we obtain $G_{k_\alpha}\in L^\infty([0,1[;\mathbb{R})$. Therefore, we have $k_\alpha\in \mathcal{K}_\omega$. Further, for $0\leq y<1$,  a simple calculation yields
$$
|F_{k'_\alpha}(y)|=1+ XE_1(X)-e^{-X},
$$
where
$$
X=\frac{1-y}{\alpha}>0.
$$
On the other hand, observe that
$$
zE_1(z)\leq e^{-z},\quad z>0.
$$ 
Therefore, we deduce that 
$$
|F_{k'_\alpha}(y)|\leq 1,\quad 0\leq y<1,
$$
which yields $F_{k'_\alpha}\in L^\infty([0,1[;\mathbb{R})$. Using a similar calculation, we get  $G_{k'_\alpha}\in L^\infty(]0,1];\mathbb{R})$. Therefore, we have $k'_\alpha\in \mathcal{K}_\omega$.
\end{proof}

\begin{lemma}\label{Laplace}
We have
$$
\left(\mathcal{L}E_1\right)(\lambda)=\frac{\ln(1+\lambda)}{\lambda},\quad \lambda>0
$$
and
$$
\left(\mathcal{L}F\right)(\lambda)=\frac{1}{\ln(1+\lambda)},\quad \lambda>0,
$$
where $\mathcal{L}$ is the Laplace transform operator and
\begin{equation}\label{FF}
F(\lambda)=\left(\int_0^\infty \frac{\lambda^{t-1}}{\Gamma(t)}\,dt\right)e^{-\lambda},\quad \lambda>0.
\end{equation}
\end{lemma}

\begin{proposition}\label{conjpn}
For every $\alpha>0$, the kernel-functions $k_\alpha$ and $k'_\alpha$ are conjugate.
\end{proposition}

\begin{proof}
Let us fix $\alpha>0$. To show that $k_\alpha$ and $k'_\alpha$ are conjugate kernels, we have to prove that 
\begin{equation}\label{apr}
\delta_{k_\alpha,k'_\alpha}(x,y)=\delta_{k'_\alpha,k_\alpha}(x,y)=1,\quad (x,y)\in \Omega.
\end{equation}
Let $(x,y)\in \Omega$ be fixed. We have
\begin{eqnarray*}
\delta_{k'_\alpha,k_\alpha}(x,y)&=&\int_y^x k'_\alpha(x,z)k_\alpha(z,y)\,dz\\
&=& \int_y^x \frac{1}{\alpha}E_1\left(\frac{x-z}{\alpha}\right) F\left(\frac{z-y}{\alpha}\right)\,dz,
\end{eqnarray*}
where $F$ is the function given by \eqref{FF}. Using the change of variable $t=\frac{z-y}{\alpha}$, we obtain
$$
\delta_{k'_\alpha,k_\alpha}(x,y)=\int_0^{\frac{x-y}{\alpha}} E_1\left(\frac{x-y}{\alpha}-t\right)F(t)\,dt,
$$
i.e.,
$$
\delta_{k'_\alpha,k_\alpha}(x,y)=(E_1*F)\left(\frac{x-y}{\alpha}\right).
$$
Since the convolution product is symmetric, we have
$$
\delta_{k'_\alpha,k_\alpha}(x,y)=\delta_{k_\alpha,k'_\alpha}(x,y).
$$
Therefore, to show \eqref{apr}, we have just to prove that
\begin{equation}\label{naim}
(E_1*F)\left(\frac{x-y}{\alpha}\right)=1.
\end{equation}
On the other hand, we have
$$
\mathcal{L}(E_1*F)(\lambda)=\mathcal{L}(E_1)(\lambda)\,\mathcal{L}(F)(\lambda),\quad \lambda>0,
$$
which gives us by Lemma  \ref{Laplace} that
$$
\mathcal{L}(E_1*F)(\lambda)=\frac{1}{\lambda},\quad \lambda>0.
$$
Note that 
$$
\frac{1}{\lambda}=\mathcal{L}(1)(\lambda),\quad \lambda>0.
$$
Therefore, we have
$$
\mathcal{L}(E_1*F)=\mathcal{L}(1),
$$
which yields \eqref{naim}.
\end{proof}

Now, we introduce the following fractional integral operators.

\begin{definition}[Fractional integrals of type (I)]
Let $\alpha>0$ and $f\in L^1([0,1];\mathbb{R})$.
The left-sided fractional integral of type (I) of order $\alpha$ of $f$ is given  by
$$
(H_0^\alpha f)(x)=\int_0^x  F\left(\frac{x-y}{\alpha}\right)f(y)\,dy,\quad \mbox{a.e. } x\in [0,1],
$$
where $F$ is given by \eqref{FF}. The right-sided fractional integral of type (I) of order $\alpha$ of $f$ is given  by
$$
(H_1^\alpha f)(x)=\int_x^1  F\left(\frac{y-x}{\alpha}\right)f(y)\,dy,\quad \mbox{a.e. } x\in [0,1].
$$
\end{definition}

\begin{definition}[Fractional integrals of type (II)]
Let $\alpha>0$ and $f\in L^1([0,1];\mathbb{R})$.
The left-sided fractional integral of type (II) of order $\alpha$ of $f$ is given  by
$$
(S_0^\alpha f)(x)=\frac{1}{\alpha}\int_0^x  E_1\left(\frac{x-y}{\alpha}\right)f(y)\,dy,\quad \mbox{a.e. } x\in [0,1].
$$
The right-sided fractional integral of type (II) of order $\alpha$ of $f$ is given  by
$$
(S_1^\alpha f)(x)=\frac{1}{\alpha}\int_x^1  E_1\left(\frac{y-x}{\alpha}\right)f(y)\,dy,\quad \mbox{a.e. } x\in [0,1].
$$
\end{definition}

Next, several properties related to the above fractional integrals can be deduced from the obtained results in Section \ref{sec3}.

\begin{proposition}\label{TH1}
For every $\alpha>0$,  
$$
H_0^\alpha, H_1^\alpha: L^1([0,1];\mathbb{R})\to L^1([0,1];\mathbb{R})
$$
are linear and continuous operators. 
\end{proposition}

\begin{proof}
First, observe that for every $\alpha>0$, we have
\begin{equation}\label{obs1H}
H_i^\alpha=I_i^{k_\alpha},\quad i=0,1.
\end{equation}
Therefore, the desired result follows from Propositions \ref{kk'}, \ref{pr1} and \ref{pr1b}.
\end{proof}

\begin{proposition}\label{TS1}
For every $\alpha>0$,  
$$
S_0^\alpha, S_1^\alpha: L^1([0,1];\mathbb{R})\to L^1([0,1];\mathbb{R})
$$
are linear and continuous operators. 
\end{proposition}

\begin{proof}
We have just to  observe that for every $\alpha>0$, we have
\begin{equation}\label{obs2S}
S_i^\alpha=I_i^{k'_\alpha},\quad i=0,1.
\end{equation}
Next, the desired result follows from Propositions \ref{kk'}, \ref{pr1} and \ref{pr1b}.
\end{proof}

\begin{remark}
Note that unlike the case of Riemann-Liouville fractional integrals, by Propositions \ref{pr2} and \ref{pr2b}, for $\alpha,\beta>0$, we have
$$
H_i^\alpha(H_i^\beta f)\not\equiv H_i^{\alpha+\beta}f,\quad i=0,1
$$
and
$$
S_i^\alpha(S_i^\beta f)\not\equiv S_i^{\alpha+\beta}f,\quad i=0,1.
$$
\end{remark}

Next, using Proposition \ref{conjpn}, we shall prove the following composition results between fractional integrals of type (I) and fractional integrals of type (II).

\begin{theorem}\label{COMPHS}
Let $\alpha>0$ and $f\in L^1([0,1];\mathbb{R})$. Then
\begin{equation}\label{r1a}
H_0^\alpha(S_0^\alpha f)(x)=S_0^\alpha(H_0^\alpha f)(x)=\int_0^x f(y)\,dy,\quad \mbox{a.e. } x\in [0,1]
\end{equation}
and
\begin{equation}\label{r2q}
H_1^\alpha(S_1^\alpha f)(x)=S_1^\alpha(H_1^\alpha f)(x)=\int_x^1 f(y)\,dy,\quad \mbox{a.e. } x\in [0,1].
\end{equation}
\end{theorem}

\begin{proof}
Using Theorem \ref{pr2}, Proposition \ref{conjpn}, \eqref{obs1H} and \eqref{obs2S},  we obtain \eqref{r1a}. Similarly, using Theorem \ref{pr2b}, Proposition \ref{conjpn}, \eqref{obs1H} and \eqref{obs2S},  we obtain \eqref{r2q}.
\end{proof}

\begin{theorem}[Integration by parts rule for fractional integrals of type (I)]\label{CHT}
Let $\alpha>0$, $f\in L^1([0,1];\mathbb{R})$ and $g\in L^\infty([0,1];\mathbb{R})$.  Then
$$
\int_0^1 (H_0^\alpha f)(x) g(x)\,dx=\int_0^1 (H_1^\alpha g)(x) f(x)\,dx.
$$
\end{theorem}

\begin{proof}
It follows immediately from Theorem \ref{INTPA} and \eqref{obs1H}.
\end{proof}

\begin{theorem}[Integration by parts rule for fractional integrals of type (II)]
Let $\alpha>0$, $f\in L^1([0,1];\mathbb{R})$ and $g\in L^\infty([0,1];\mathbb{R})$.  Then
$$
\int_0^1 (S_0^\alpha f)(x) g(x)\,dx=\int_0^1 (S_1^\alpha g)(x) f(x)\,dx.
$$
\end{theorem}

\begin{proof}
It follows immediately from Theorem \ref{INTPA} and \eqref{obs2S}.
\end{proof}

Further, we shall prove the following approximation result.

\begin{theorem}\label{appT}
Let $f\in L^1([0,1]; \mathbb{R})$.  Then 
$$
\lim_{\alpha\to 0^+}\| S_0^\alpha f-f\|_{L^1([0,1];\mathbb{R})}=0.
$$
\end{theorem}

Before giving the proof of Theorem \ref{appT}, we need some preliminaries on Harmonic Analysis (see, for example \cite{CH}).

\begin{definition}[Approximate identity]
An approximate identity is a family
$\{h_\alpha\}_{\alpha>0}$ of real valued functions such that
\begin{itemize}
\item[\rm{(i)}] $h_\alpha\geq 0$, for all $\alpha>0$.
\item[\rm{(ii)}]$\int_{\mathbb{R}}h_\alpha(t)\,dt=1$, for all $\alpha>0$.
\item[\rm{(iii)}] For every $\delta>0$, 
$$
\lim_{\alpha\to 0^+} \int_{|x|>\delta} h_\alpha(t)\,dt =0.
$$
\end{itemize}
\end{definition}

\begin{lemma}\label{HAAA}
Let $h$ be a real valued function such that 
\begin{itemize}
\item[\rm{(i)}] $h\geq 0$.
\item[\rm{(ii)}]$\int_{\mathbb{R}}h(t)\,dt=1$.
\end{itemize}
Then the family $\{h_\alpha\}_{\alpha>0}$ given by
$$
h_\alpha(t)=\frac{1}{\alpha}h\left(\frac{t}{\alpha}\right),\quad t\in \mathbb{R},\,\alpha>0,
$$
is an approximate identity. 
\end{lemma}

\begin{lemma}\label{HAAA2}
Let $\{h_\alpha\}_{\alpha>0}$ be an approximate identity. If $g\in L^1(\mathbb{R};\mathbb{R})$, then
$$
g*h_\alpha\in L^1(\mathbb{R};\mathbb{R}),\quad \alpha>0.
$$
Moreover, we have
$$
\lim_{\alpha\to 0^+} \|g*h_\alpha-g\|_{L^1(\mathbb{R};\mathbb{R})}=0.
$$
\end{lemma}

Let $\widetilde{E_1}: \mathbb{R}\to \mathbb{R}$ be the function given by
\begin{eqnarray*}
\widetilde{E_1}(x)=\left\{ \begin{array}{lll}
E_1(x),  &&\mbox{if } x>0,\\
0, &&\mbox{ otherwise.}
\end{array}
\right.
\end{eqnarray*}
Observe that  the function $\widetilde{E_1}$ satisfies the assumptions of Lemma \ref{HAAA}. Therefore, the family $\{G_\alpha\}_{\alpha>0}$ given by 
\begin{eqnarray*}
G_\alpha(x)=\left\{ \begin{array}{lll}
\frac{1}{\alpha}E_1\left(\frac{x}{\alpha}\right),  &&\mbox{ if } x>0,\\
0, &&\mbox{ otherwise,}
\end{array}
\right.
\end{eqnarray*}
is an approximate identity. Hence, by Lemma \ref{HAAA2}, we have the following result.
\begin{lemma}\label{HAAA3}
For every $g\in L^1(\mathbb{R};\mathbb{R})$, we have
$$
\lim_{\alpha\to 0^+} \|g*G_\alpha-g\|_{L^1(\mathbb{R};\mathbb{R})}=0.
$$
\end{lemma}

Now, we are ready to prove  Theorem \ref{appT}.

\begin{proof}
Let $f\in L^1([0,1];\mathbb{R})$. Then $F_f\in L^1(\mathbb{R};\mathbb{R})$, where $F_f$ is the function given by
\begin{eqnarray*}
F_f(x)=\left\{ \begin{array}{lll}
f(x),  &&\mbox{a.e. } x\in [0,1],\\
0, &&\mbox{ otherwise.}
\end{array}
\right.
\end{eqnarray*}
By Lemma \ref{HAAA3}, we have
$$
\lim_{\alpha\to 0^+} \|F_f*G_\alpha-F_f\|_{L^1(\mathbb{R};\mathbb{R})}=0,
$$
which yields
\begin{equation}\label{aya}
\lim_{\alpha\to 0^+} \|(F_f*G_\alpha)_{|[0,1]}-f\|_{L^1([0,1];\mathbb{R})}=0.
\end{equation}
On the other hand, observe that 
\begin{equation}\label{aya2}
(F_f*G_\alpha)_{|[0,1]}=S_0^\alpha f \,\mbox{ in }\, L^1([0,1];\mathbb{R}).
\end{equation}
Combining \eqref{aya} with \eqref{aya2}, the desired result follows.
\end{proof}

Using a similar argument as above, we obtain the following approximation result for $S_1^\alpha$.

\begin{theorem}\label{appT2}
Let $f\in L^1([0,1]; \mathbb{R})$.  Then 
$$
\lim_{\alpha\to 0^+}\| S_1^\alpha f-f\|_{L^1([0,1];\mathbb{R})}=0.
$$
\end{theorem}

Further, we introduce the following fractional derivative operators.

\begin{definition}\label{asd}
Let $\theta>1$. Let $f\in L^1([0,1];\mathbb{R})$ be  such that $S_0^{\theta-1} f$  has  an absolutely continuous representative.  The left-sided fractional derivative of order $\theta$ of $f$ is given by
$$
\left(\mathcal{D}_0^\theta f\right)(x)=\frac{d}{dx} \left(S_0^{\theta-1} f\right)(x),\quad \mbox{ a.e. } x\in [0,1].
$$ 
Let $f\in L^1([0,1];\mathbb{R})$ be such that $S_1^{\theta-1} f$  has  an absolutely continuous representative. The right-sided fractional derivative of order $\theta$ of $f$ is given by
$$
\left(\mathcal{D}_1^\theta f\right)(x)=-\frac{d}{dx} \left(S_1^{\theta-1} f\right)(x),\quad \mbox{ a.e. } x\in [0,1].
$$ 
\end{definition}

\begin{remark}
Note that by Proposition \ref{conjpn}, for all $\theta>1$, we have
$$
k_{\theta-1} \in \mbox{conj}\left(k'_{\theta-1}\right).
$$
Moreover, under the assumptions of Definition \ref{asd}, we have
\begin{equation}\label{derv1}
\left(\mathcal{D}_0^\theta f\right)(x)=\left(D_0^{k'_{\theta-1}}f\right)(x), \quad \mbox{ a.e. } x\in [0,1]
\end{equation}
and
\begin{equation}\label{derv2}
\left(\mathcal{D}_1^\theta f\right)(x)=\left(D_1^{k'_{\theta-1}}f\right)(x), \quad \mbox{ a.e. } x\in [0,1].
\end{equation}
\end{remark}

Using Theorem \ref{TN1}, we obtain the following result.

\begin{theorem}\label{NFDI}
Let $\theta>1$. For every $f\in L^1([0,1];\mathbb{R})$, we have
$$
\mathcal{D}_0^\theta (H_0^{\theta-1} f)(x)= f(x),\quad \mbox{ a.e. }x\in [0,1]
$$
and
$$
\mathcal{D}_1^{\theta}(H_1^{\theta-1} f)(x)= 
f(x),\quad \mbox{ a.e. }x\in [0,1].
$$
\end{theorem}

The following result follows from Theorem \ref{TN3}.

\begin{theorem}\label{KATR}
Let $\theta>1$ and $f\in L^1([0,1];\mathbb{R})$ be such that $S_0^{\theta-1} f$  has an absolutely continuous representative. Then
$$
H_0^{\theta-1}\left(\mathcal{D}_0^\theta f\right)(x)=f(x)-\left(S_0^{\theta-1} f\right)(0) F\left(\frac{x}{\theta-1}\right),\quad \mbox{ a.e. } x\in [0,1],
$$
where $F$ is given by \eqref{FF}.
\end{theorem}

Similarly, using Theorem \ref{TN3b}, we obtain the following result.

\begin{theorem}\label{KATR}
Let $\theta>1$ and $f\in L^1([0,1];\mathbb{R})$ be such that $S_1^{\theta-1} f$  has an absolutely continuous representative. Then
$$
H_1^{\theta-1}\left(\mathcal{D}_1^\theta f\right)(x)=f(x)-\left(S_1^{\theta-1} f\right)(1) F\left(\frac{1-x}{\theta-1}\right),\quad \mbox{ a.e. } x\in [0,1].
$$
\end{theorem}

The next results characterize the conditions for the existence of the fractional derivatives $\mathcal{D}_0^\theta$ and $\mathcal{D}_1^\theta$ in the space $AC([0,1];\mathbb{R})$.

\begin{theorem}\label{TYYRG}
Let $\theta>1$. If $f\in AC([0,1];\mathbb{R})$, then $\mathcal{D}_0^\theta f$ exists almost everywhere in $[0,1]$, and can be represented in the form
\begin{equation}\label{sim}
\left(\mathcal{D}_0^\theta f\right)(x)=f(0) \left[\frac{1}{\theta-1} E_1\left(\frac{x}{\theta-1}\right)\right]+ \int_0^x \frac{1}{\theta-1}E_1\left(\frac{x-y}{\theta-1}\right)\frac{df}{dy}(y)\,dy
,\quad \mbox{ a.e. } x\in [0,1].
\end{equation}
\end{theorem}

\begin{proof}
Let $f\in AC([0,1];\mathbb{R})$. Then
$$
f(x)=f(0)+\int_0^x \frac{df}{dy}(y)\,dy,\quad 0\leq x\leq 1.
$$
Using \eqref{r1a}, we obtain
$$
f(x)=f(0)+H_0^{\theta-1}\left(S_0^{\theta-1} \frac{df}{dx}\right)(x),\quad \mbox{a.e. } x\in [0,1],
$$
which yields
$$
(S_0^{\theta-1} f)(x)=(S_0^{\theta-1} f(0))(x)+S_0^{\theta-1} H_0^{\theta-1}\left(S_0^{\theta-1} \frac{df}{dx}\right)(x),\quad \mbox{a.e. } x\in [0,1].
$$
On the other hand, after some calculations, we obtain
$$
\left(S_0^{\theta-1} f(0)\right)(x)=f(0)\left[\left(\frac{x}{\theta-1} \right)E_1\left(\frac{x}{\theta-1}\right)-e^{\frac{-x}{\theta-1}}+1\right],\quad \mbox{ a.e. } x\in [0,1].
$$
Therefore,
$$
(S_0^{\theta-1} f)(x)=f(0)\left[\left(\frac{x}{\theta-1} \right)E_1\left(\frac{x}{\theta-1}\right)-e^{\frac{-x}{\theta-1}}+1\right]+S_0^{\theta-1} H_0^{\theta-1}\left(S_0^{\theta-1} \frac{df}{dx}\right)(x),\quad \mbox{a.e. } x\in [0,1].
$$
Again, using \eqref{r1a}, we obtain
$$
(S_0^{\theta-1} f)(x)=f(0)\left[\left(\frac{x}{\theta-1} \right)E_1\left(\frac{x}{\theta-1}\right)-e^{\frac{-x}{\theta-1}}+1\right]+\int_0^x \left(S_0^{\theta-1} \frac{df}{dy}\right)(y)\,dy,\quad \mbox{a.e. } x\in [0,1].
$$
Taking the derivative with respect to $x$, \eqref{sim} follows.
\end{proof}

Similarly, we have the following result for the right-sided fractional derivative $\mathcal{D}_1^{\theta}$.

\begin{theorem}\label{TYYRGb}
Let $\theta>1$. If $f\in AC([0,1];\mathbb{R})$, then $\mathcal{D}_1^\theta f$ exists almost everywhere in $[0,1]$, and can be represented in the form
$$
\left(\mathcal{D}_1^\theta f\right)(x)=f(1) \left[\frac{1}{\theta-1} E_1\left(\frac{1-x}{\theta-1}\right)\right]-\int_x^1 \frac{1}{\theta-1}E_1\left(\frac{y-x}{\theta-1}\right)\frac{df}{dy}(y)\,dy
,\quad \mbox{ a.e. } x\in [0,1].
$$
\end{theorem}

Next, using Theorems \ref{appT} and \ref{TYYRG}, we shall prove the following approximation result.

\begin{theorem}\label{fmfg}
Let $f\in AC([0,1];\mathbb{R})$ be such that $f(0)=0$. Then
$$
\lim_{\theta\to 1^+}\left\|\mathcal{D}_0^\theta f -\frac{df}{dx}\right\|_{L^1([0,1];\mathbb{R})}=0.
$$
\end{theorem}

\begin{proof}
Let $f\in AC([0,1];\mathbb{R})$ be such that $f(0)=0$. Using Theorem \ref{TYYRG}, we obtain
$$
\left\|\mathcal{D}_0^\theta f -\frac{df}{dx}\right\|_{L^1([0,1];\mathbb{R})}=\left\|S_0^{\theta-1} \frac{df}{dx}-\frac{df}{dx}\right\|_{L^1([0,1];\mathbb{R})},\quad \theta>1.
$$
Passing to the limit as $\theta\to 1^+$ and using Theorem \ref{appT}, the desired result follows.
\end{proof}

Similarly, using Theorems \ref{appT2} and \ref{TYYRGb}, we obtain the following  result.

\begin{theorem}\label{fmfgb}
Let $f\in AC([0,1];\mathbb{R})$ be such that $f(1)=0$. Then
$$
\lim_{\theta\to 1^+}\left\|\mathcal{D}_1^\theta f -\left(-\frac{df}{dx}\right)\right\|_{L^1([0,1];\mathbb{R})}=0.
$$
\end{theorem}

We end this section with the following  fractional integration by parts rule. Just before, let us introduce the following functional spaces. Given $\alpha>0$, we say that $f\in H_0^\alpha(L^1)$ if there exists $\varphi\in L^1([0,1];\mathbb{R})$ such that $f(x)=(H_0^\alpha \varphi)(x)$, a.e. $x\in [0,1]$. We say that $f\in H_1^\alpha(L^\infty)$ if there exists $\varphi\in L^\infty([0,1];\mathbb{R})$ such that $f(x)=(H_1^\alpha \varphi)(x)$, a.e. $x\in [0,1]$.

\begin{theorem}\label{RIPGD}
Let $\theta>1$. If $f\in H_1^{\theta-1}(L^\infty)$ and $g\in H_0^{\theta-1}(L^1)$, then
$$
\int_0^1 f(x) \left(\mathcal{D}_0^\theta g\right)(x)\,dx=\int_0^1 \left(\mathcal{D}_1^\theta f\right)(x)g(x)\,dx.
$$
\end{theorem}

\begin{proof}
Let $f\in H_1^{\theta-1}(L^\infty)$ and $g\in H_0^{\theta-1}(L^1)$. By the definition of the functional space $H_1^{\theta-1}(L^\infty)$, there exists a certain function $\varphi_f\in L^\infty([0,1];\mathbb{R})$ such that 
\begin{equation}\label{eqff}
f(x)=\left(H_1^{\theta-1} \varphi_f\right)(x),\quad \mbox{ a.e. } x\in [0,1].
\end{equation}
Similarly, by the definition of the functional space $H_0^{\theta-1}(L^1)$, there exists a certain function $\varphi_g\in L^1([0,1];\mathbb{R})$ such that 
\begin{equation}\label{eqgg}
g(x)=\left(H_0^{\theta-1} \varphi_g\right)(x),\quad \mbox{ a.e. } x\in [0,1].
\end{equation}
Further, using \eqref{eqff} and Theorem \ref{NFDI}, we obtain
\begin{equation}\label{eqff2}
\varphi_f(x)=(\mathcal{D}_1^\theta f)(x),\quad \mbox{a.e. } x\in [0,1].
\end{equation}
Using \eqref{eqgg} and Theorem \ref{NFDI}, we obtain
\begin{equation}\label{eqgg2}
(\mathcal{D}_0^\theta g)(x)=\varphi_g(x),\quad \mbox{a.e. } x\in [0,1].
\end{equation}
Next, using \eqref{eqff} and \eqref{eqgg2}, we have
$$
\int_0^1 f(x) \left(\mathcal{D}_0^\theta g\right)(x)\,dx=\int_0^1 \left(H_1^{\theta-1} \varphi_f\right)(x) \varphi_g(x)\,dx.
$$
Applying Theorem \ref{CHT}, we obtain
\begin{equation}\label{ouf}
\int_0^1 f(x) \left(\mathcal{D}_0^\theta g\right)(x)\,dx=\int_0^1 (H_0^{\theta-1} \varphi_g)(x) \varphi_f(x)\,dx.
\end{equation}
Finally, using \eqref{eqgg}, \eqref{eqff2} and \eqref{ouf}, the desired result follows.
\end{proof}

\section{An existence result for a boundary value problem involving $k'$-derivative}\label{sec6}

In this section, we discuss the solvability of a boundary value problem  involving $k'$-derivative. More precisely, we are concerned with the boundary value problem
\begin{equation}\label{bvp}
\begin{aligned}
&(D_0^{k'}u)(t)=f(t,u(t)),\quad \mbox{ a.e. } t\in [0,1],\\
&(I_0^{k'}u)(0)=0,
\end{aligned}
\end{equation} 
where $(k,k')\in \mathcal{K}_\omega\times \mathcal{K}_\omega$ (with $[a,b]=[0,1]$) are two conjugate kernels, $u\in C([0,1];\mathbb{R})$ is such that $I_0^{k'}u$ admits an absolutely continuous representative, and $f: [0,1]\times \mathbb{R}\to \mathbb{R}$ is a continuous function. 

The functional space $C([0,1];\mathbb{R})$ is equipped with the Chebyshev norm
$$
\|y\|_\infty=\max\{|y(t)|:\, 0\leq t\leq 1\},\quad y\in C([0,1];\mathbb{R}).
$$

Problem \eqref{bvp} is investigated under the following assumptions:
\begin{itemize}
\item[(A1)] $I_0^k \left(C([0,1];\mathbb{R})\right)\subset C([0,1];\mathbb{R})$, i.e., $I_0^k y\in C([0,1];\mathbb{R})$, for all $y\in C([0,1];\mathbb{R})$.
\item[(A2)] For all $(t,x_1,x_2)\in [0,1]\times \mathbb{R}\times \mathbb{R}$, we have 
$$
|f(t,x_1)-f(t,x_2)|\leq c_f |x_1-x_2|,
$$
where $c_f$ is a positive constant such that $c_f\|I_0^k 1\|_\infty<1$.
\end{itemize}

Suppose now that $u\in C([0,1];\mathbb{R})$ is a solution to \eqref{bvp}. Using Theorem \ref{TN3}, we obtain 
$$
u(t)=\left(I_0^k f(\cdot,u(\cdot))\right)(t),\quad \mbox{ a.e. } t\in [0,1].
$$
By continuity, we may identify both functions appearing in the above equation, so that we get
\begin{equation}\label{bvp2}
u(t)=\left(I_0^k f(\cdot,u(\cdot))\right)(t),\quad t\in [0,1].
\end{equation}
Conversely, suppose that $u\in C([0,1];\mathbb{R})$ is a solution to \eqref{bvp2}. By Theorem \ref{pr2}, we have
$$
(I_0^{k'}u)(t)=\left(I_0^{\delta_{k',k}}f(\cdot,u(\cdot))\right)(t),\quad \mbox{ a.e. } t\in [0,1].
$$
Since $k$ and $k'$ are two conjugate kernels, we obtain
$$
(I_0^{k'}u)(t)=\int_0^t f(s,u(s))\omega(s)\,ds,\quad \mbox{ a.e. } t\in [0,1],
$$
which yields
\begin{eqnarray*}
(D_0^{k'}u)(t)&=&\left(\frac{1}{\omega(t)} \frac{d}{dt}\right)\int_0^t f(s,u(s))\omega(s)\,ds\\
&=& f(t,u(t)),\quad \mbox{ a.e. } t\in [0,1].
\end{eqnarray*}
Moreover, by identification of $I_0^{k'}u$ with its absolutely continuous representative 
$$
t\mapsto \int_0^t f(s,u(s))\omega(s)\,ds,
$$
we have 
$$
(I_0^{k'}u)(0)=0.
$$
Therefore, $u\in C([0,1];\mathbb{R})$ is a solution to \eqref{bvp}.

Consequently,  we deduce that solving \eqref{bvp} is equivalent to solve \eqref{bvp2} in $C([0,1];\mathbb{R})$. 
Next, let us introduce the mapping $T: C([0,1];\mathbb{R})\to C([0,1];\mathbb{R})$ given by
$$
(Tu)(t)=\left(I_0^k f(\cdot,u(\cdot))\right)(t),\quad u\in C([0,1];\mathbb{R}),\, t\in [0,1].
$$
ote that from (A1) and the continuity of $f$, the mapping $T$ is well-defined. Further, let us fix $(u,v)\in C([0,1];\mathbb{R})\times C([0,1];\mathbb{R})$.  Taking into account 
assumptions (A1) and (A2), for all $t\in [0,1]$, we have
\begin{eqnarray*}
|(Tu)(t)-(Tv)(t)| &=& \left|\left(I_0^k f(\cdot,u(\cdot))\right)(t)-\left(I_0^k f(\cdot,v(\cdot))\right)(t)\right|\\
&=&\left|\int_0^t k(t,s)f(s,u(s))\omega(s)\,ds-\int_0^t k(t,s)f(s,v(s))\omega(s)\,ds\right|\\
&=&\left|\int_0^t k(t,s) \left(f(s,u(s))-f(s,v(s))\right)\omega(s)\,ds\right|\\
&\leq & c_f \left(\int_0^t k(t,s) \omega(s)\,ds\right) \|u-v\|_\infty\\
&=& c_f (I_0^k1)(t)\|u-v\|_\infty\\
&\leq & C \|u-v\|_\infty,
\end{eqnarray*}
where $C=c_f \|I_0^k1\|_\infty$. Hence, we obtain
$$
\|Tu-Tv\|_\infty\leq C \|u-v\|_\infty,\quad (u,v)\in C([0,1];\mathbb{R})\times C([0,1];\mathbb{R}).
$$
Moreover, taking in consideration assumption (A2), we have $0\leq C<1$. Finally, Banach contraction principle implies that the mapping $T$ has one and only one fixed point $u^*\in C([0,1];\mathbb{R})$, which is the unique solution to \eqref{bvp}. Hence, the following result holds.

\begin{theorem}\label{TBVP}
Under  assumptions (A1) and (A2),  \eqref{bvp} admits one and only one solution $u^*\in C([0,1];\mathbb{R})$.
\end{theorem}

\bibliographystyle{amsplain}

\end{document}